\documentclass[11pt]{amsart}

\usepackage{amsfonts}
\usepackage{amssymb}
\usepackage{amsthm}
\usepackage{amsmath}
\usepackage{hyperref}
\hypersetup{
    linktoc=page,  
    colorlinks=true,       
    linkcolor=red,          
    citecolor=green,        
    filecolor=magenta,      
    urlcolor=cyan           
}
\usepackage{tikz}
\usepackage{pgf}
\usepackage{graphicx}
\usepackage[top=3.5cm, bottom=3.5cm, left=2.5cm, right=2.5cm]{geometry}

\title{Polynomials and harmonic functions on discrete groups}
\date{}
\author{Tom Meyerovitch, Idan Perl, Matthew Tointon and Ariel Yadin}

\address{Tom Meyerovitch,  Idan Perl  and Ariel Yadin:
Department of Mathematics, Ben Gurion University of the Negev, Be'er Sheva, Israel.}
\email{\{mtom, perli, yadina\} @math.bgu.ac.il}
\address{ Matthew Tointon: Laboratoire de Math\'ematiques, Universit\'e Paris-Sud 11, 91405 Orsay cedex, France}
\email{matthew.tointon@math.u-psud.fr}

\newtheorem{prop}{Proposition}[section]
\newtheorem{theorem}[prop]{Theorem}
\newtheorem{lemma}[prop]{Lemma}
\newtheorem{corollary}[prop]{Corollary}
\newtheorem{conjecture}[prop]{Conjecture}

\theoremstyle{definition}
\newtheorem{definition}[prop]{Definition}

\theoremstyle{remark}

\newtheorem{example}[prop]{Example}

\newtheorem{remark}[prop]{Remark}

\theoremstyle{theorem}

\newcommand{\R}{\mathbb{R}}
\newcommand{\C}{\mathbb{C}}

\newcommand{\N}{\mathbb{N}}
\newcommand{\Z}{\mathbb{Z}}

\newcommand{\E}{\mathbb{E}}

\newcommand{\F}{\mathbb{F}}
\newcommand{\n}{\mathfrak{n}}

\newcommand{\dist}{\text{\textup{dist}}}

\newcommand{\Span}{\text{\textup{span}}}

\newcommand{\Hom}{\text{\textup{Hom}}}

\newcommand{\poly}{\text{\textup{poly}}}

\newcommand{\p}{\partial}

\newcommand{\IP}[1]{\langle #1 \rangle}

\numberwithin{equation}{section}

\usepackage{capt-of}

\begin{document}
\maketitle
\begin{abstract}
Alexopoulos proved that on a finitely generated virtually nilpotent group, the restriction of a harmonic function of polynomial growth to a torsion-free nilpotent subgroup of finite index is always a polynomial in the Mal'cev coordinates of that subgroup.
For general groups, vanishing of higher-order discrete derivatives gives a natural notion of polynomial maps, which has been considered by Leibman and others.
We provide a simple proof of  Alexopoulos's result using this notion of polynomials, under the weaker hypothesis that the space of harmonic functions of polynomial growth of degree at most $k$ is finite dimensional.
We also prove that for a finitely generated group the Laplacian maps the polynomials of degree $k$ surjectively onto the polynomials of degree $k-2$.
We then present some corollaries.
In particular, we  calculate precisely the dimension of the space of harmonic functions of polynomial
growth of degree at most $k$ on a virtually nilpotent group, extending an old result of  Heilbronn for the abelian case, and refining a more recent result of Hua and Jost.

\end{abstract}

\setcounter{tocdepth}{1}
\tableofcontents

\section{Introduction}
This paper has two principal aims. The first is to give a new proof of a result of Alexopoulos describing the structure of the harmonic functions of polynomial growth on a virtually nilpotent group.
The second is to compute precisely the dimensions of the subspaces of harmonic functions of polynomial growth of given degree.
This computation is in turn an ingredient in a recent proof of the third author that the dimension of the space of \emph{all} harmonic functions is finite if and only if $G$ is virtually cyclic \cite{tointon}.

Before presenting a precise formulation of our results, let us briefly give a non-rigorous overview. The structure theorem of Alexopoulos shows that a harmonic function of polynomial growth on a virtually nilpotent group is essentially a \emph{polynomial} \cite{alex}. A study of the former therefore necessarily entails a study of the latter. There are a few natural ways to define a polynomial, each of which is useful in certain contexts. For a general finitely generated group, a particularly natural approach is to define
a polynomial as a function that vanishes after taking a bounded number of derivatives. Following Leibman \cite{leibman} and others, we take this as our principal definition; see Definition \ref{def:poly} below,

Another possible way to define a polynomial on a group, as used by Alexopoulos in his original argument, is in terms of a certain \emph{coordinate system}, defined in terms of the lower central series of the group. We describe this in detail in Section \ref{sec:coords}. If one is comfortable with the theory of Lie groups, then in the context of torsion-free nilpotent groups
a natural coordinate system to choose 
is one given by a {\em Mal'cev} basis; see Subsection \ref{subsec:alex}.

As it turns out all of these definitions are equivalent. We prove this precisely as Proposition \ref{thm:polys.equiv} below, using an argument essentially due to Leibman. Throughout this paper, we exploit this equivalence to choose the definition that appears best suited to proving each result.

\subsection{The structure of harmonic functions of polynomial growth}
Let $G$ be a finitely generated group.
Fix a finite generating set $S$ for $G$, and assume that $S$ is \emph{symmetric}, in the sense that $s \in S$ if and only if $s^{-1} \in S$.
This gives rise to a metric structure (in fact a graph structure) on $G$, defined via
$$ \dist(x,y) = \dist_{G,S}(x,y) = |x^{-1} y| = \min \{ n \ : \ x^{-1} y = s_1 \cdots s_n \ , \ s_j \in S \} . $$
This metric is easily seen to be the graph metric of the {\em Cayley graph} of $G$ with respect to $S$,
obtained by taking the vertex set to be elements of $G$ and edges defined by the relation $x \sim y \iff x^{-1}y \in S$.
By definition, this metric is left invariant, so all information is contained in the distance to the identity element $\dist(x,1) = |x|$.

For a function $f:G \to \C$ and a non-negative integer $k$, define the (possibly infinite) quantity $|| f ||_k$ via
$$ || f ||_k = \limsup_{r \to \infty} r^{-k} \max_{|x| \leq r } |f(x)| . $$
This quantity is invariant under the natural $G$-action on functions defined by $xf(y) = f(x^{-1}y)$.
Note also that $||f ||_k < \infty$ if and only if there exists a constant $C=C_f >0$ such that
for all $x\ne1$ we have $|f(x)| \leq C |x|^k$.

In the case that $|| f ||_k<\infty$ we say that $f$ has \emph{polynomial growth of degree at most $k$}. This does not depend on a particular choice of finite symmetric generating set $S$. It is straightforward to check that $|| \ \cdot \ ||_k$ defines a $G$-invariant seminorm on the space of functions of polynomial growth of degree at most $k$.

Now consider a probability measure $\mu$ on $G$.
Throughout, we assume that $\mu$ satisfies the following standing assumptions:
\begin{itemize}
\item $\mu$ is {\em symmetric} in the sense that $\mu(x) = \mu(x^{-1})$;
\item $\mu$ is {\em adapted} to $G$, which is to say that the support of $\mu$ generates $G$ as a group;
\item $\mu$ has {\em exponential tail},  which is to say that there exists some $\varepsilon>0$ such that $\mu(x) \le e^{-\varepsilon |x|}$.
\end{itemize}

The uniform probability measure on a finite, symmetric generating set is a canonical example of a measure satisfying these standing assumptions. 
Considering the larger class of measures with exponential tail enables us in particular to move freely
to finite-index subgroups (see Proposition \ref{prop:dim}, below), although in some of our results we do make the stronger assumption that $\mu$ is finitely supported.
Note that if $\mu$ has exponential tail
with respect to some finite symmetric generating set
then it also has exponential tail with respect to any other finite symmetric generating set, so
the notion of having exponential tail does not depend on the specific choice of generating set.
The same standing assumptions appeared in  \cite{mey-yad}, where measures with exponential tail were called ``smooth''.

The \emph{Laplacian operator} $\Delta= \Delta_\mu$ acting on functions $f:G\to \C$  of polynomial growth is defined via
$$
\Delta f (x) =f(x) -\E_s[f(xs)]= \sum_s \mu(s) (f(x) - f(xs) ) .
$$
Note that for $f$ of polynomial growth of degree $k$  the assumption that $\mu$ has exponential tail implies that
$$\sum_s \mu(s) |f(xs)| \le C \sum_s \mu(s) |s|^k < \infty.$$
Thus the sum implicit in the expression $\Delta f (x)$ converges absolutely.

A function $f:G \to \C$ is {\em $\mu$-harmonic}, or just {\em harmonic} when $\mu$ is clear from context,  if  $\Delta_\mu f(x)= 0$ for all $x \in G$, or equivalently if
$$ f(x) = \sum_{s \in G} \mu(s) f(xs) . $$

In this paper we are mainly interested in the spaces $H^k(G,\mu)$ of harmonic functions of polynomial growth, defined by
$$ H^k(G,\mu) = \{ f : G \to \C \ \big| \ || f ||_k < \infty \ , \ \Delta_\mu f =0\} . $$

Leibman \cite{leibman} has studied in some detail the notion of a polynomial mapping between arbitrary groups. The first purpose of this paper is to describe harmonic functions of polynomial growth in terms of such polynomial mappings. We use the following definition, which is Leibman's definition specialised to our setting.
\begin{definition}[Polynomial]\label{def:poly}
Given $f:G\to\C$ and an element $u\in G$ we define the \emph{left derivative} $\p_uf$ of $f$ with respect to $u$ by
$\p_uf(x)=f(u x)-f(x)$; that is, $\p_u f = u^{-1}f-f$, where $u^{-1}f$ is the left action of $u^{-1}$ on $f$.

Let $H < G$ be a subgroup. A function  $f:G \to \C$ is called a \emph{polynomial with respect to $H$} if there exists some integer $k\ge0$ such that
\[ \p_{u_1}\cdots\p_{u_{k+1}}f = 0\text{ for all }u_1,\ldots,u_{k+1}\in H \]
The \emph{degree} (with respect to $H$) of a non-zero polynomial $f$ is the smallest such $k$.
When $H=G$ we simply say that  $f:G\to\C$ is a \emph{polynomial}.
We denote the space of polynomials of degree at most $k$ by $P^k(G)$.
For notational convenience we also define $P^k(G)=\{0\}$ for $k<0$.
\end{definition}

Recall
that for $x,y \in G$ the commutator  is given by $[x,y] = x^{-1} y^{-1} x y$.
For subgroups $H_1, H_2 \leq G$, $[H_1,H_2]$ denotes the subgroup generated by commutators:
$$[H_1,H_2] := \langle[h_1,h_2] ~:~ h_1 \in H_1,~ h_2 \in H_2 \rangle.$$
The lower central series of a group $G$ is defined inductively by $G_1:=G, G_{k+1}:=[G_1,G_k]$. A group  $N$ is called {\em nilpotent} if there exists $k\in\N$ such that $N_{k+1}=\{1\}$. 
The minimal such $k$ is called the \emph{step} or \emph{nilpotency class} of $N$.

For a group property $\mathcal{P}$ we say that a group $G$ is {\em virtually $\mathcal{P}$} if there exists a finite-index subgroup $H < G$ such that $H$ is $\mathcal{P}$.  For example, $G$ is {\em virtually nilpotent} if there exists a finite-index nilpotent subgroup $H$ in $G$.

In his paper \cite{alex}, Alexopoulos proves various results concerning random walks on a finitely generated virtually nilpotent group $G$, including a central limit theorem and a Berry--Esseen-type theorem. His methods include embedding a finite-index torsion-free nilpotent subgroup of $N$ as a lattice in a Lie group.

One consequence of Alexopoulos's work that is particularly useful for studying harmonic functions of polynomial growth is a description of such functions in terms of polynomials, as follows.
\begin{theorem}[Alexopoulos \cite{alex}]\label{thm:alex.cosets}
Let $G$ be a finitely generated group with a finite-index nilpotent subgroup $N$,
and let $\mu$ be a probability measure on $G$ that satisfies the standing assumptions and is, in addition, finitely supported.
Suppose that $f:G\to \C$ is $\mu$-harmonic and has polynomial growth of degree at most $k$. Then $f$ is a polynomial of degree at most $k$ with respect to $N$. Moreover, $N_{k+1}$ acts trivially on $f$ from the left, in the sense that $gf=f$ for every $g\in N_{k+1}$.
\end{theorem}
The first aim of this paper is to give a more direct and elementary proof of Theorem \ref{thm:alex.cosets}, in particular without appealing to the theory of Lie groups. In fact, our argument recovers essentially the same conclusion under a technically weaker assumption. The first step in our proof of Theorem \ref{thm:alex.cosets} is to show that if $G$ is virtually nilpotent then there is \emph{some} finite-index subgroup $H$ of $G$ such that every function in $H^k(G,\mu)$ is a polynomial with respect to $H$. However, for this part of the argument the only way in which we use the virtual nilpotency of $G$ is in using the well-known fact that this implies $\dim H^k(G,\mu)<\infty$. Thus, the first step of our argument yields the following generalisation of Theorem \ref{thm:alex.cosets}.
\begin{theorem}\label{thm:alex.gen}
Let $G$ be a finitely generated group and let $\mu$ be a probability measure on $G$ satisfying the standing assumptions. Let $k\ge0$ and suppose that $\dim H^k(G,\mu) <\infty$.
Then there is a finite-index normal subgroup $H$ of $G$ such that every $f\in H^k(G,\mu)$ factors through $G/H_{k+1}$ and is a polynomial of degree $k$ with respect to $H$.
\end{theorem}

\begin{remark}
Theorem \ref{thm:alex.gen} is a generalisation of Theorem \ref{thm:alex.cosets} in the sense that its hypotheses are implied by those of Theorem \ref{thm:alex.cosets}. However, it is conjectured that finite dimensionality of $H^k(G,\mu)$ implies that $G$ is virtually nilpotent \cite{mey-yad}, and if this conjecture held then the hypotheses of Theorems \ref{thm:alex.cosets} and \ref{thm:alex.gen} would be equivalent and Theorem \ref{thm:alex.gen} would not be significantly stronger than Theorem \ref{thm:alex.cosets}. The first and last authors have proved this conjecture in the cases of virtually solvable and linear groups \cite{mey-yad}, a fact that we use in our proof of Theorem \ref{thm:alex.gen}.
\end{remark}

\subsection{The dimensions of the spaces $H^k(G,\mu)$}
One of our motivations for studying the structure of the spaces $H^k(G,\mu)$ is to compute their dimensions. An additional result that is key to being able to do this is the following theorem regarding the image of the Laplacian operating on the space of polynomials.
\begin{theorem}\label{thm:lap_poly_img}
Let $G$ be a finitely generated group, and let $\mu$ be a probability measure on $G$ satisfying the standing assumptions.  Then for all $k \in \N$ we have $$\Delta_\mu P^{k}(G)=P^{k-2}(G).$$
\end{theorem}
In conjunction with the theorems of the previous subsection, this allows us to arrive at the following precise computation of $\dim H^k(G,\mu)$ for an arbitrary virtually nilpotent group $G$.
\begin{theorem}\label{thm:dim}
Let $G$ be a finitely generated group with a finite-index nilpotent subgroup $N$, and let $\mu$ be a
probability measure on $G$ that satisfies the standing assumptions and is finitely supported. Then $\dim H^k(G,\mu)=\dim P^k(N)-\dim P^{k-2}(N)$.
\end{theorem}
\begin{remark}
Some of the material we use to prove Theorem \ref{thm:dim} originally appeared in an early version of the third author's paper \cite{tointon}, in which the bound $\dim H^k(G,\mu)\ge\dim P^k(N)-\dim P^{k-1}(N)$ was obtained (the notation here differs slightly from the notation there).
\end{remark}
\begin{remark}
The space $P^k(G)$ is finite-dimensional for an arbitrary finitely generated group $G$ \cite[Proposition 1.15]{leibman}, and so the quantity $\dim P^k(N)-\dim P^{k-2}(N)$ appearing in Theorem \ref{thm:dim} is well defined.
\end{remark}

%

In conjunction with the results of \cite{mey-yad}, this gives the following, which we prove in Subsection \ref{subsec:gg}.
\begin{corollary}\label{cor:dim.indep.of.mu}
Let $G$ be a finitely generated group and let $\mu,\nu$ be two probability measures on $G$ that satisfy the standing assumptions. Suppose that $\dim H^k(G,\mu) < \infty$ and $\dim H^k(G,\nu) < \infty$. Then
$$ \dim H^k(G,\mu) =  \dim H^k(G,\nu).$$
\end{corollary}

It is possible to obtain a much more explicit expression for $\dim P^k(G)$. Polynomials on an arbitrary group have a more explicit description in terms of certain coordinate systems. We review this equivalent description in Section \ref{sec:coords}. From this description, it follows that
the dimension of  $P^k(G)$ for a finitely generated group $G$ is an explicit function of the ranks of the abelian quotients coming from its lower central series, as follows.
\begin{prop}\label{prop:dim.of.poly}
Let $G$ be a finitely generated group,
and let $d_j$ denote the torsion-free rank of the abelian group $G_{j}/G_{j+1}$.
Then $\dim P^k(G)$ is equal to the number of non-negative integer solutions $(x_{j,t} \ : \ j=1,\ldots,k,t=1,\ldots, d_j)$ to the 
inequality
\begin{equation}\label{eq:integers_sols_dim}
\sum_{j=1}^k j\sum_{t=1}^{d_j}x_{j,t} \leq k.
\end{equation}
\end{prop}
In combination with Theorem \ref{thm:dim} this implies a precise computation of $\dim H^k(G,\mu)$ when $G$ is virtually nilpotent. To our knowledge, this is the first such computation that covers the cases in which $G$ is not abelian.
\begin{corollary}
If $G$ is a finitely generated group with finite-index nilpotent subgroup $N$ and $\mu$ is a finitely supported probability measure on $G$ satisfying the standing assumptions then $\dim H^k(G,\mu)$ is determined by the ranks of the abelian quotients $N_{j}/N_{j+1}$.
\end{corollary}
The values of $P^k(G)$ implied by Proposition \ref{prop:dim.of.poly} could in principle be given in explicit form. For example, in the case where $G$ is $2$-step nilpotent, which is to say that $G_3=\{1\}$, and hence $d_3=0$, on setting $y = \sum_{t=1}^{d_2} x_{2,t}$ and summing over possible values of $y$ we obtain
\[
\dim P^k(G) = \sum_{y=0}^{\lfloor k/2 \rfloor} { d_2-1+y \choose d_2-1 } { d_1 + k-2y \choose d_1 }. 
\]

When $G$ is an abelian group of torsion-free rank $d$, Proposition \ref{prop:dim.of.poly} implies that $\dim P^k(G)$ is equal to the number of non-negative integer solutions to the equation
$\sum_{j=1}^d x_j \leq k$, which is equal to ${d+k \choose d}$. Theorem \ref{thm:dim} is therefore a generalisation of an old result of Heilbronn \cite{heilbronn}, who in the 1940s used Theorem \ref{thm:lap_poly_img} in the case that $G=\Z^d$ and $\mu$ was the uniform probability measure on the standard symmetric generating set to show that
\[
\dim H^k(\Z^d,\mu)= {d+k \choose d} - {d+k-2 \choose d}.
\]
Hua, Jost \& X. Li-Jost \cite{hjlj} have shown that this remains true if $\Z^d$ is replaced by any abelian group having $\Z^d$ as a finite-index subgroup,
and $\mu$ is the uniform probability measure on \emph{any} finite symmetric generating set.

More generally, Theorem \ref{thm:dim} recovers and refines another relatively recent result of Hua \& Jost \cite{hj} stating that when $G$ has a finite-index nilpotent subgroup $N$ we have
\begin{equation}\label{eq:hua-jost}
\dim H^k(G,\mu) \leq C k^{D-1}
\end{equation}
for $k\ge1$ and some constant $C>0$ depending on $\mu$, where $D$ is the polynomial growth rate of $G$ (also called the \emph{homogeneous dimension} of $G$).  The Bass--Guivarc'h formula \cite{bass,guivarc'h} states that the homogeneous dimension is given by
\begin{equation}\label{eq:bg}
D=\sum_{i=1}^\ell i d_i,
\end{equation}
where $d_i$ is the torsion-free rank of the abelian quotient $N_i/N_{i+1}$. In Subsection \ref{subsec:hj} we prove the following bounds in terms of the \emph{rank} $d$ of $N$, defined by
\[
d=\sum_{i=1}^\ell d_i.
\]
\begin{corollary}\label{cor:hj}
Let $G$ be a finitely generated group with a finite-index nilpotent subgroup $N$ of rank $d$, and let $\mu$ be a finitely supported probability measure on $G$ satisfying the standing assumptions. Then there exist constants $c_d,C_d>0$ such that
\[
c_dk^{d-1} \leq \dim H^k(G,\mu) \leq C_dk^{d-1}
\]
for $k\ge1$.
\end{corollary}
In particular, this implies the following.
\begin{corollary}
The bound (\ref{eq:hua-jost}) is sharp as $k\to\infty$ if and only if $G$ is virtually abelian.
\end{corollary}

\subsection{More corollaries}
If $N$ is a nilpotent group then it is well known that the set $t(N)$ of elements of finite order in $N$ forms a normal subgroup, called the \emph{torsion subgroup} of $N$,
and that $N/t(N)$ is then a torsion-free nilpotent group \cite[\S5.2]{Robinson}.
We prove the following corollary in Subsection \ref{subsec:kernel}.
\begin{corollary}\label{cor:kernel}
Let $G$ be a finitely generated group with a nilpotent subgroup $N$ of finite index such that $N/t(N)$ is of nilpotency class exactly $\ell$.
Let $\mu$ be a finitely supported probability measure on $G$ satisfying the standing assumptions. Let $k\in\N$, and let $K$ be the kernel of the natural action of $G$ on $H^k(G,\mu)$. Then $K$ is finite if and only if $k\ge \ell$.
\end{corollary}

In \cite{mey-yad}, the  first and last authors discussed the conjecture that $\dim H^k(G,\mu)< \infty$ implies that $G$ is virtually nilpotent. Corollary \ref{cor:dim.indep.of.mu} implies that a consequence of this conjecture would be that $\dim H^k(G,\mu)$ was independent of $\mu$ for every group $G$ and every finitely supported measure $\mu$ satisfying the standing assumptions. This would also follow from the following essentially weaker conjecture.
\begin{conjecture}
Let $G$ be a finitely generated group. Then either $\dim H^k(G,\mu)$ is infinite for all probability measures $\mu$ satisfying the standing assumptions, or it is finite for all such measures.
\end{conjecture}

\subsection{Acknowledgements}
The authors are grateful to Emmanuel Breuillard for helpful conversations. T.M.  would like to acknowledge funding from the People Programme (Marie Curie Actions) of the European Union's Seventh Framework Programme (FP7/2007-2013) under REA grant agreement no. 333598,
and from the Israel Science Foundation (grant no. 626/14). 
I.P.  is supported by a Negev Fellowship from the Kreitman School of Ben-Gurion University of the Negev.
M.T.  is supported by ERC grant GA617129 `GeTeMo'.
A.Y and I.P.  are also partially supported by 
Grant no.\ 2010357 from the United States--Israel
Binational Science Foundation (BSF).

\section{Polynomials on groups}

\subsection{Basic facts about polynomials}
We record some basic properties of polynomials on groups, some of which can be found in  \cite{leibman}.
The group $G$ acts on functions by left and right translation: for $x,y\in G$ and $f:G\to \C$, left translation is defined by
$L_x f(y)=xf(y)=f(x^{-1} y)$, and right translation by $R_xf(y)=f^x(y)=f(yx^{-1})$. Define the {\em left derivative} by
$\partial_x f=x^{-1}f-f$, and the {\em right derivative} $\p^x f=f^{x^{-1}}-f$. Note that the left and right actions of $G$ on $\C^G$ commute, and that left differentiation and right differentiation therefore commute in the sense that $\p_x \p^y= \p^y \p_x$ for all $x,y\in G$.
We use this property extensively without mention.

Leibman observed that choosing left or right derivatives in Definition \ref{def:poly} does not change the set of polynomials.
\begin{prop}[Leibman {\cite[Corollary 2.13]{leibman}}]
\label{prop:def.equiv}
Let $f:G\to \C$. For all $k\geq 1$, the following are equivalent.
\begin{enumerate}
\renewcommand{\labelenumi}{(\arabic{enumi})}
\item For every $x_1,...,x_{k+1}\in G$ we have $\p_{x_1} \cdots \p_{x_{k+1}} f = 0$ (that is,  $f\in P^k(G)$).
\item For every $x_1,...,x_{k+1}\in G$ we have $\p^{x_1} \cdots \p^{x_{k+1}} f = 0$.
\end{enumerate}
\end{prop}

The following is an immediate consequence of Proposition \ref{prop:def.equiv} and the fact that left differentiation and the right group action commute and vice versa.
\begin{corollary}
\label{cor:P^k.G.invariant}
The space $P^k(G)$ is invariant under both the left and right group actions.
\end{corollary}
We will often make use of the following identities, each of which is readily verifiable by direct computation.
\begin{align}
\label{eq:deriv.by.prod}
\p_{xy}f &= y^{-1}\p_xf + \p_yf \\
\label{eq:deriv.by.comm}
\p_{[x,y]}f &= \p_{y}\p_{x}yxf - \p_{x}\p_{y}yxf \\
\label{eq:product.rule}
\p_x (fh) &=x^{-1}f\cdot\p_xh+\p_xf\cdot h
\end{align}

\begin{lemma}
\label{lem:deg.prod}
If $f\in P^k(G)$ and $h\in P^m(G)$ then $fh\in P^{k+m}(G)$.
\end{lemma}
\begin{proof}
This follows from \eqref{eq:product.rule} and Corollary \ref{cor:P^k.G.invariant} by induction on $k+m$.
\end{proof}
\begin{remark}
The stronger statement that $\deg fh= \deg f+\deg h$ follows from (\ref{eq:hom.deg.prod}) and Proposition \ref{thm:polys.equiv}, below.
\end{remark}

If $N$ is a normal subgroup of $G$ and $f:G\to\C$ is such that $f(nx)=f(x)$ for every $n\in N$ then we say that $f$ \emph{factors through} the quotient $G/N$, or that $N$ \emph{acts trivially} on $f$. If $f$ factors through $G/N$ then, writing $\pi$ for the projection $G\to G/N$, there exists $\hat f:G/N\to\C$ such that $f=\hat f\circ\pi$. Note that if $f$ factors through $G/N$ then so do $\p_xf$ and $\p^xf$.
\begin{lemma}
\label{lem:normal.factor}
Let $G$ be a group with $N\lhd G$, and suppose that $f: G \to \C$ factors through $G/N$. Then, in the above notation, we have
\[
\p_{xN}\hat f=\widehat{\p_xf}
\]
for every $x\in G$. In particular, $\hat f\in P^k(G/N)$ if and only if $f \in P^k(G)$.
\end{lemma}

\begin{proof}
For every $x,y\in G$ we have $\p_{xN} \hat f(yN) = f(x y) - f(y) = \p_x f (y) = \widehat{\p_x f}(y N)$.
\end{proof}

\begin{lemma}[Leibman {\cite[Proposition $2.15$]{leibman}}]\label{lem:comm.deriv}
Let $G$ be a group and let $f \in P^k(G)$. Then $f$ factors through $G/G_{k+1}$.
\end{lemma}

\subsection{Growth of polynomials}

The main purpose of this subsection is to prove the following result.
\begin{prop}\label{prop:growth.of.polys}
Let $f:G\to\C$ be a polynomial, and let $k$ be a non-negative integer. If $f$ is of degree at most $k$ then $|| f ||_k < \infty$. Conversely, if
$|| f ||_{k+1} = 0$ then $f$ is of degree at most $k$.
\end{prop}

\begin{lemma}\label{lem:deriv.hom}
Let $k\in\N$. For each $f\in P^k(G)$ the function $\varphi_f:G^k\to\C$ defined by
\[
\varphi_f(x_1,\ldots,x_k)=\p_{x_1}\cdots\p_{x_k}f(1)
\]
is a homomorphism in each variable $x_j$.
\end{lemma}
\begin{proof}
The $k=1$ case holds because $P^1(G)$ is just the space of homomorphisms into the additive group $\C$
plus constants. For $k>1$ we proceed by induction.
It follows from \eqref{eq:deriv.by.prod} that
\[
\begin{split}
\p_{x_1}\cdots\p_{x_{j-1}}\p_{u_jv_j}\p_{x_{j+1}}\cdots\p_{x_k}f(1)
\qquad\qquad\qquad\qquad\qquad\qquad\qquad\qquad\qquad\qquad\qquad\qquad\\
=\p_{x_1}\cdots\p_{x_{j-1}}v_j^{-1}\p_{u_j}\p_{x_{j+1}}\cdots\p_{x_k}f(1)
+\p_{x_1}\cdots\p_{x_{j-1}}\p_{v_j}\p_{x_{j+1}}\cdots\p_{x_k}f(1).
\end{split}
\]
Corollary \ref{cor:P^k.G.invariant} implies that $v_j^{-1}\p_{u_j}\p_{x_{j+1}}\cdots\p_{x_k}f\in P^{j-1}$, and so by induction on $k$ we may assume that the quantity $\p_{x_1}\cdots\p_{x_{j-1}}v_j^{-1}\p_{u_j}\p_{x_{j+1}}\cdots\p_{x_k}f(1)$ is a homomorphism in each of the variables $x_1,\ldots,x_{j-1}$.
Write $y^x=x^{-1}yx$, and note that $\p_yx=x\p_{y^x}$. Since homomorphisms into the additive group $\C$ are invariant under conjugation, this implies that
\[
\begin{split}
\p_{x_1}\cdots\p_{x_{j-1}}\p_{u_jv_j}\p_{x_{j+1}}\cdots\p_{x_k}f(1)
\qquad\qquad\qquad\qquad\qquad\qquad\qquad\qquad\qquad\qquad\qquad\qquad\\
=v_j^{-1}\p_{x_1}\cdots\p_{x_{j-1}}\p_{u_j}\p_{x_{j+1}}\cdots\p_{x_k}f(1)
+\p_{x_1}\cdots\p_{x_{j-1}}\p_{v_j}\p_{x_{j+1}}\cdots\p_{x_k}f(1),
\end{split}
\]
and so the result follows from the fact that $\p_{x_1}\cdots\p_{x_{j-1}}\p_{u_j}\p_{x_{j+1}}\cdots\p_{x_k}f$ is constant.
\end{proof}

\begin{proof}[Proof of Proposition \ref{prop:growth.of.polys}]
The first direction is by induction on $n = \deg f$.
If $n=0$ then $f$ is constant and thus $||f ||_0 < \infty$.
If $n>0$ then, by induction, there exists a constant $C>0$ such that $|\p^s f (y) | \leq C |y|^{n-1}$ for every $s$ with $|s|=1$.
Fixing $x \in G$, and writing $x = s_1 \cdots s_{|x|}$ with $|s_i|=1$, we have
$$ f(x) = f(1) + \sum_{j=1}^{|x|} \p^{s_j} f(s_1 \cdots s_{j-1}) , $$
and so
\begin{align*}
|f(x)| &\leq |f(1) | + C \sum_{j=1}^{|x|} (j-1)^{n-1} = O(|x|^n).
\end{align*}

To prove the converse direction we assume that $|| f ||_{k+1} = 0$,
which implies that $|f(x)| = o(|x|^{k+1})$ as $|x| \to \infty$.
Still writing $n=\deg f$,
we conclude from Lemma \ref{lem:deriv.hom} that
\[
\p_{x_1^m} \ldots \p_{x_n^m}f(1) = m^n  \p_{x_1}\cdots \p_{x_n}f(1).
\]
In particular, this implies that, for fixed $x_1,\ldots,x_n$, as $m\to\infty$ we have
\begin{align}\label{eq:n<k}
m^n|\p_{x_1}\cdots\p_{x_n}f(1)| & = |\p_{x_1^m}\cdots\p_{x_n^m}f(1)|
\leq  \sum_{S \subset \{1,\ldots,n\} } \Big| f \Big( {\textstyle\prod_{j \in S} x_j^m } \Big) \Big|
=o(m^{k+1}). 
\end{align}
By definition of degree there exist $x_1,\ldots,x_n$ such that $\p_{x_1}\cdots\p_{x_n}f(1)\ne0$,
and then letting $m \to \infty$ in \eqref{eq:n<k} implies that $n\le k$.
\end{proof}

\section{Harmonic functions of polynomial growth are polynomials}
\label{sec:harmonic}
The purpose of this section is to prove Theorem \ref{thm:alex.gen}.
We start by recalling some basic facts about groups acting by linear transformations.
Suppose that $H$ is a group acting linearly on an $n$-dimensional vector space $V$ over a field $\F$.
We denote by $\mathrm{Hom}(H,\F^\times)$ the characters of the group $H$ into the multiplicative group $\F^\times$.
Given $\lambda \in \mathrm{Hom}(H,\F^\times)$,
we denote the {\em weight space} corresponding to $\lambda$
by
$$ V_\lambda = V_{\lambda}^{(1)} = \{ v \in V \ : \ x v = \lambda(x) v \ , \ \forall \ x \in H\} =
\bigcap_{x \in H} \ker (x-\lambda(x) I) . $$
The {\em $k$-th generalised weight space} is defined inductively by
$$ V_{\lambda}^{(k)} = \{ v \in V \ : \ (x-\lambda(x)I) v \in V_{\lambda}^{k-1} \ , \ \forall \ x \in H\} . $$
We also set $V_{\lambda}^{(0)} = \{0\}$, which is consistent with these definitions.
The {\em generalised weight space} is defined by
$$ V_{\lambda}^* = \bigcup_k V_{\lambda}^{(k)} . $$
Thus, $v \in V_{\lambda}^*$ if and only if there exists $k$ such that $(x-\lambda(x) I)^k v = 0$
for all $x \in H$.  Note that $V_{\lambda}^*$ is an $H$-invariant subspace.
It is a well-known fact from linear algebra that
$V_{\lambda}^* \cap V_{\beta}^* = \{ 0\}$ if $\lambda \neq \beta$.
It is important to note that this definition is with respect to some group acting linearly on $V$, and depends on the specific choice of the acting group.
If $G$ acts linearly on $V$ and $H \leq G$  then $V_{\lambda}^*$ with respect to $G$
is a subspace of $V_{\lambda}^*$ with respect to $H$.

If $G$ acts linearly on a vector space $V$ and $K$ is the kernel of this action then $G/K$ is isomorphic to a subgroup of $GL(V)$. If $G/K$ is virtually nilpotent then we say the action of $G$ on $V$ is virtually nilpotent. We make use of the following standard lemma about virtually nilpotent linear actions; for a proof see \cite[Lemma 6.2]{mey-yad}, for example.

\begin{lemma}\label{lem:eigen}
Let $G$ be a group, and let $V$ be a finite-dimensional vector space over an algebraically closed field $\F$ such that $G$ acts linearly on $V$ and such that this action is virtually nilpotent. Then there exists a finite-index normal subgroup $H\lhd G$ with respect to which we have
\[
V=\bigoplus_{j=1}^r V_{\lambda_j}^\ast,
\]
with $\lambda_1,\ldots,\lambda_r\in\Hom(H,\F^\times)$.
\end{lemma}

\begin{proof}[Proof of Theorem \ref{thm:alex.gen}]
The case $k=0$ is handled by observing that the constants are the only finite dimensional $G$-invariant space of bounded harmonic functions; see \cite{mey-yad} for details and references.
A proof for the case $k=1$ appears in \cite{mey-yad}.

The group $G$ from Theorem \ref{thm:alex.gen} acts linearly on $H^k(G,\mu)$ via $g\cdot h(x)=h(g^{-1}x)$, and since $\dim H^k(G,\mu)<\infty$ this defines a homomorphism $\pi:G \to GL_n(\C)$ for some $n\in\N$.
A function $f:G\to\C$ belongs to $H^k(G,\mu)$ if and only if it is of the form $f= \tilde f\circ\pi$, with $\tilde f \in H^k(\pi(G),\mu \circ \pi^{-1})$.
It follows that $\pi(G) < GL_n(\C)$ is a linear group with $\dim H^k(\pi(G),\mu \circ \pi^{-1})<\infty$. Since $\mu$ satisfies the standing assumptions, so too does the measure $\mu \circ \pi^{-1}$ on $\pi(G)$, and so \cite[Theorem 1.4]{mey-yad} implies that $\pi(G)$ is virtually nilpotent. Lemma \ref{lem:eigen} therefore implies that there is a finite-index subgroup $H$ of $G$ with respect to which we can decompose $V=H^k(G,\mu)$ as $H^k(G,\mu)=\bigoplus_{j=1}^r V_{\lambda_j}^\ast$, with $\lambda_j\in\Hom(H,\C^\times)$.

Fix some $\lambda=\lambda_j$. 
Let $f\in V_\lambda^{(1)}$. Then for every $x\in H$ we have $f(x^{-n})=\lambda(x)^nf(1)$, which, since $f$ is bounded by a polynomial, implies that $|\lambda(x)|=1$. 
If $g_1,\ldots,g_t$ is a complete set of right-coset representatives of $H$ in $G$, this implies that $|f(hg_i)|=|f(g_i)|$ for every $h\in H$ and every $i$, and so $f$ is bounded on $G$.
 The Liouville property for virtually nilpotent groups (see, for example, \cite{kai-ver}) therefore implies that $f$ is constant on $G$. This implies first of all that $V_\lambda^\ast = \{0\}$ unless $\lambda$ is the trivial character $1$, and so in fact we have $H^k(G,\mu) = V_1^\ast$.

Now note that $f \in V_1^{(n)}$ if and only if for all $x \in H$ we have $\p_x f = x^{-1} f-f \in V_1^{(n-1)}$. 
Since $V_1^{(0)} = \{0\}$, for every $n$ this implies that if $f:G\to\C$ belongs to $V_1^{(n)}$ then $f|_H\in P^n(H)$. In particular, every $f\in H^k(G,\mu)$ satisfies $f|_H\in P^n(H)$ for $n=\dim H^k(G,\mu)$. However, by definition every such $f$ also satisfies $\|f\|_k<\infty$. 
Given a finite symmetric generating set $S'$ of $H$, there exists $C>1$ so that for any $x,y \in H$ we have
$$C^{-1} \dist_{G,S}(x,y) < \dist_{H,S'}(x,y)  < C \dist_{G,S}(x,y).$$
It follows that restricting  to a finite index subgroup does not increase the degree of polynomial growth, hence $\|f|_H\|_k<\infty$. So in fact $f|_H\in P^k(H)$ by Proposition \ref{prop:growth.of.polys}.

To see that $f$ is a polynomial of degree $k$ with respect to $H$ on all of $G$, we must show more generally that for every $t\in G$ the function $H\to\R$ defined by $x\mapsto f(xt)$ is a polynomial of degree $k$. The function $p_t=t^{-1}f$ belongs to $H^k(G,\mu)$, and so by the first part of the proof it restricts to a polynomial of degree at most $k$ on $H$. In particular, since $H$ is normal, for every $x,u_1,\ldots,u_{k+1}\in H$ we have $\partial_{t^{-1}u_1t}\cdots\partial_{t^{-1}u_{k+1}t}p_t(t^{-1}xt)=0$, and so the function $q_t:H\to\R$ defined by $q_t(x)=p_t(t^{-1}xt)$ is a polynomial of degree at most $k$ on $H$. However, $q_t(x)=p_t(t^{-1}xt)=t^{-1}f(t^{-1}xt)=f(xt)$, and so $f$ is indeed a polynomial with respect to $H$ on $G$. Moreover, $q_t$ factors through $H/H_{k+1}$ by Lemma \ref{lem:comm.deriv}, and so $f$ factors through $G/H_{k+1}$, as required.
\end{proof}

\section{Polynomials in terms of group coordinates}\label{sec:coords}

\subsection{Generalised commutator subgroups}

The purpose of this section is to give a more explicit description of the space $P^k(G)$ of polynomials on a finitely generated group $G$, in terms of certain coordinate systems on $G$. We define these coordinate systems in terms of certain subgroups of $G$ called the \emph{generalised commutator subgroups}, denoted $\overline G_i$ and defined by
\begin{equation}\label{eq:gen.com}
\overline G_i=\{c\in G:(\exists n)(c^n\in G_i)\}.
\end{equation}
\begin{lemma}\label{lem:tors.finite}
The generalised commutator subgroups are all characteristic subgroups. Moreover, if $G$ is finitely generated then $\overline G_i$ contains $G_i$ as a finite-index subgroup.
\end{lemma}
\begin{proof}
The set of elements of finite order in a nilpotent group is a subgroup, called the \emph{torsion subgroup} \cite[\S5.2]{Robinson}. The set $\overline G_i$ is precisely the pre-image in $G$ of the torsion subgroup of the nilpotent quotient $G/G_i$, and so it is a group. It is also clearly invariant under automorphisms of $G$.

If $G$ is finitely generated then so is $G/G_i$, and since $G/G_i$ is nilpotent this implies that each of its subgroups is also finitely generated. In particular, being a nilpotent group generated by finitely many finite-order elements, its torsion subgroup is finite. This readily implies that $G_i$ has finite index in $\overline G_i$.
\end{proof}

One reason for introducing the generalised commutator subgroups $\overline G_i$ is to handle issues arising in \cite[3.14 \& 3.15]{leibman} when some of the factors $G_i/G_{i+1}$ have torsion. See Proposition \ref{prop:leib.3.14}, below.

The subgroups $\overline G_i$ appear quite naturally in the study of polynomials, thanks to the following refinement of Lemma \ref{lem:comm.deriv}.
\begin{prop}\label{prop:comm.deriv}
Let $G$ be a group and let $f\in P^k(G)$. Then for every $c\in\overline G_i$ we have
$\partial_c f\in P^{k-i}(G)$ and $\p^c f \in P^{k-i}(G)$. In particular, $f$ factors through $G/\overline G_{k+1}$.
\end{prop}
\begin{lemma}\label{lem:torsion.deriv}
Let $G$ be a group and let $f\in P^k(G)$. Let $c\in G$ be an element of finite order. Then 
$\p_cf = 0$.
\end{lemma}
\begin{proof}
The proof is by induction on $k$.
If $k=0$ then the lemma is trivial. If $k\ge1$ then by induction we have $\partial_c\partial_cf=0$, from which it easily follows that for each $x\in G$ and $n\in\N$ we have $f(c^nx)=f(x)+n\partial_cf(x)$
(because $f(c^{j+2} x)-f(c^{j+1} x) = f(c^{j+1} x) - f(c^j x)$ for all integers $j$). 
Taking $n\ne0$ such that $c^n=1$ therefore implies that $\partial_cf(x)=0$, and so the lemma is proved.
\end{proof}

\begin{lemma}\label{lem:fact.through.tor.free.nilp}
Let $G$ be a group and let $f\in P^k(G)$. Then $f$ factors through $G/\overline G_{k+1}$.
\end{lemma}
\begin{proof}
By Lemmas \ref{lem:normal.factor} and \ref{lem:comm.deriv} we may assume that $G_{k+1}=\{1\}$, and so the result follows from Lemma \ref{lem:torsion.deriv}.
\end{proof}

\begin{lemma}\label{lem:gen.coms}
Let $G$ be a group, and suppose that $x\in\overline G_i$ and $y\in\overline G_j$. Then $[x,y]\in\overline G_{i+j}$.
\end{lemma}
\begin{proof}
We prove the more general statement that if $\alpha$ is a \emph{commutator form of weight $r$} in the sense of \cite[Definition 3.2]{nilp.frei}, and if $x_1,\ldots,x_r$ satisfy $x_i\in\overline G_{j_i}$, then $\alpha(x_1,\ldots,x_r)\in\overline G_{\sum_{i=1}^rj_i}$. We prove the claim first in the case that $G$ is nilpotent. Writing $l$ for the nilpotency class of $G$, we proceed by induction on $l-r$, noting that when $r>l$ the required statement is trivial, since $\alpha(x_1,\ldots,x_r)=1$ in that case.

Let $m$ be such that $x_i^m\in G_{j_i}$ for each $i$, and note therefore that
\begin{equation}\label{eq:comm.form}
\alpha(x_1^m,\ldots,x_r^m)\in G_{\sum_{i=1}^rj_i}.
\end{equation}
It follows directly from \cite[Proposition B.2]{nilp.frei} that there are commutators $\eta_1,\ldots,\eta_n$ of weight greater than $r$ in the $x_i$ in the sense of \cite[Definition 3.1]{nilp.frei}, with each $x_i$ featuring in each $\eta_t$ at least once, such that
\begin{equation}\label{eq:comm.form2}
\alpha(x_1,\ldots,x_r)^{m^r}=\alpha(x_1^m,\ldots,x_r^m)\eta_1\cdots\eta_n.
\end{equation}
Since each $x_i$ features in each $\eta_t$ at least once, the induction hypothesis implies that each $\eta_t\in\overline G_{\sum_{i=1}^rj_i}$, and so (\ref{eq:comm.form}) and (\ref{eq:comm.form2}) combine to imply that $\alpha(x_1,\ldots,x_r)^{m^r}\in\overline G_{\sum_{i=1}^rj_i}$, and hence that $\alpha(x_1,\ldots,x_r)\in\overline G_{\sum_{i=1}^rj_i}$, as claimed.

The claim in the general case follows from the claim in the nilpotent case since the group $G/\overline G_{\sum_{i=1}^rj_i}$ is nilpotent.
\end{proof}

\begin{proof}[Proof of Proposition \ref{prop:comm.deriv}]
The first claim is that for every $c\in\overline G_i$ and every $k$ we have $\p_c(P^k(G))\subset P^{k-i}(G)$. We prove this by induction on $k-i$, noting that when $i>k$ it is immediate from Lemma \ref{lem:fact.through.tor.free.nilp}.

Let $x\in G$ and $c\in\overline G_i$, and observe that \eqref{eq:deriv.by.comm} and Corollary \ref{cor:P^k.G.invariant} imply that
\begin{align*}
\p_x\p_c(P^k(G))   &\subset \p_c\p_x(P^k(G)) + \p_{[x,c]}(P^k(G)) \\
    &\subset \p_c(P^{k-1}(G)) + \p_{[x,c]}(P^k(G)).
\end{align*}
Lemma \ref{lem:gen.coms} implies that $[x,c] \in \overline G_{i+1}$, so the induction hypothesis implies that $\p_x\p_c(P^k(G))\subset P^{k-i-1}(G)$. Since $x$ was arbitrary, this implies that $\p_c(P^k(G))\subset P^{k-i}(G)$, and the claim is proved.

The proof for $\p^c$ is similar.
\end{proof}

\subsection{Coordinate systems on finitely generated groups}
Let $G$ be a finitely generated group. Lemma \ref{lem:gen.coms} implies that the groups $\overline G_i/\overline G_{i+1}$ are all torsion-free abelian; write $d_i$ for the rank of $\overline G_i/\overline G_{i+1}$, and define $n_i=d_1+\ldots+d_i$.

\begin{definition}\label{def:coord.syst}
Let $G$ be a finitely generated group. For each $i\in\N$ let $e_{n_{i-1}+1},\ldots,e_{n_i}$ be elements whose images in $G/\overline G_{i+1}$ form a basis for $\overline G_i/\overline G_{i+1}$. Then for each $k\in\N$, every element of $G$ can be represented, modulo $\overline G_{k+1}$, by a unique expression of the form
\begin{equation}\label{eq:coords.def}
x\overline G_{k+1}=e_1^{x_1}\cdots e_{n_k}^{x_{n_k}}\overline G_{k+1}
\end{equation}
with $x_i\in\Z$. Moreover, the value of each $x_i$ is independent of the choice of $k$, and so this defines, for each $x\in G$, a unique (possibly finite) sequence $x_1,x_2,\ldots$.

We call $e_1,e_2,\ldots$ a \emph{coordinate system} on $G$.
For each $x\in G$, we define the sequence $x_1,x_2,\ldots$ in the expression (\ref{eq:coords.def}) to be the \emph{coordinates of $x$} with respect to the coordinate system $e_1,e_2,\ldots$.
\end{definition}
\begin{definition}[Coordinate polynomial on $G$]\label{def:coord.poly}
Let $G$ be a finitely generated group with coordinate system $e_1,e_2,\ldots$. Then a \emph{coordinate monomial} on $G$ with respect to $e_1,e_2,\ldots$ is a function $q:G\to\C$ of the form
\[
q(x)=\lambda x_1^{a_1}\cdots x_r^{a_r},
\]
with $\lambda\in\C$, $r\le n_k$, each $a_i$ a non-negative integer, and $x_1,x_2,\ldots,x_r$ the coordinates of $x$ given by \eqref{eq:coords.def}.
A \emph{coordinate polynomial} on $G$ with respect to $e_1,e_2,\ldots$ is a finite sum of coordinate monomials.

For each $e_i$ we define $\sigma(i)=\sup\{k\in\N:e_i\in\overline G_k\}$. We then define the \emph{degree} of the monomial $q(x)=\lambda x_1^{a_1}\cdots x_r^{a_r}$ by $\deg q=\sigma(1)a_1+\cdots+\sigma(r)a_r$.
If $q_1,\ldots,q_t$ are monomials then we define the \emph{degree} of the polynomial $p=q_1+\cdots+q_t$ by $\deg p=\max_i\deg q_i$.
\end{definition}

\begin{example}\label{ex:heis}
Let $H$ be the ($3$-dimensional) Heisenberg group over $\Z$, defined by 
$$ H = H_3(\Z) = \left\{ \IP{x,y,z} \ | \ x,y,z \in \Z \right\} \qquad \IP{x,y,z} := \left[  
\begin{smallmatrix}
1 & x & z \\
0 & 1 & y \\
0 & 0 & 1
\end{smallmatrix}
\right] , $$
with the operation of matrix multiplication. The group $H$ is finitely generated, by $\{\IP{1,0,0},\IP{0,1,0}\}$, for example. The only degree-$0$ coordinate polynomials on any group are the constant functions. The degree-$1$ coordinate polynomials on $H$ are spanned by the functions mapping $\IP{x,y,z}$ to
\[
x \ , \ y.
\]
The degree-$2$ coordinate polynomials are spanned by the functions mapping $\IP{x,y,z}$ to
$$ x^2 \ , \ y^2 \ , \ xy \ , \ z . $$
The degree-$3$ coordinate polynomials are spanned by the functions mapping $\IP{x,y,z}$ to
$$ x^3 \ , \ y^3 \ , \ x^2 y \ , \ x y^2 \ , \ xz \ , \ yz . $$
\end{example}

Note that the degree of coordinate polynomials behaves similarly under multiplication to the standard degree of polynomials on $\R^d$, in that if $p,p':G\to\R$ are coordinate polynomials
 then their pointwise product $pp'$ is also a coordinate polynomial on $G$, and satisfies
\begin{equation}\label{eq:hom.deg.prod}
\deg pp'=\deg p+\deg p'.
\end{equation}

\begin{prop}[Leibman]\label{thm:polys.equiv}
Let $G$ be a finitely generated group with coordinate system $e_1,e_2,\ldots$ and suppose that $f:G\to\C$. Then $f\in P^k(G)$ if and only if $f$ is a coordinate polynomial of degree at most $k$.
\end{prop}
In combination with the main results of this paper, Proposition \ref{thm:polys.equiv} allows us to describe quite explicitly the spaces $H^k(G,\mu)$ for a nilpotent group $G$. For example, the reader may care to check that, on the Heisenberg group $H_3(\Z)$ defined in Example \ref{ex:heis}, the degree-$2$ coordinate polynomials mapping $\IP{x,y,z}$ to
$$  1 \ , \ x \ , \ y \ , \ x^2-y^2 \ , \ xy \ , \ z $$
are harmonic with respect to the uniform measure $\mu$ on the symmetric generating set
\[
\{\IP{\pm1,0,0},\IP{0,\pm1,0}\},
\]
and linearly independent. Theorems \ref{thm:alex.cosets} and \ref{thm:lap_poly_img}, and Proposition \ref{thm:polys.equiv}, therefore imply that these polynomials form a basis for the space $H^2(H_3(\Z),\mu)$.

Without the statement about degrees, Proposition \ref{thm:polys.equiv} is almost a direct consequence of \cite[Proposition $3.12$]{leibman}.
The statement about the degrees of the coordinate maps  is closely related to \cite[Proposition $3.14$]{leibman}.
However, our definition of coordinates involves the subgroups $\overline{G}_i$ rather than the usual lower central series, in order to handle the possibility that $G_i/G_{i+1}$ has torsion,
and so we give the proof in full here.

As a step in the proof of Proposition \ref{thm:polys.equiv} we prove the following result.
\begin{prop}\label{prop:coord.mon.poly}
Let $G$ be a finitely generated group with coordinate system $e_1,e_2,\ldots$. Then the coordinate map $G\to\R$ defined by $x\mapsto x_i$ belongs to $P^{\sigma(i)}(G)$.
\end{prop}

At this point it is convenient to recall the definition of \emph{polynomial mappings} between general groups, following Leibman \cite{leibman}. The definition is the natural generalisation of the derivative definition of complex-valued polynomials on a group. Specifically, for groups $H,G$ and a map $f:H\to G$ one defines $\p^h f(x)= f(xh)(f(x))^{-1}$ for every $x,h\in H$, and then the space $P^k(H,G)$ of polynomial mappings of degree at most $k$ from $H$ to $G$ is defined by
\[
P^k(H,G)=\{f:H\to G \ \big| \ \p^{h_1}\cdots\p^{h_{k+1}}f(x)=1\text{ for all $x,h_1,\ldots,h_k\in H$}\}.
\]
Note that $P^k(H)=P^k(H,\C)$.

For polynomial mappings between groups there is also a finer notion of degree, which is not a single parameter. Leibman considered the notion of ``lc-degree'' for polynomial mappings taking values in nilpotent groups.
In order to capture this finer notion, Green and Tao \cite{gt} introduced a slightly more general formulation of polynomial mappings with respect to a \emph{filtration}. If $G$ is a nilpotent group then a \emph{filtration} $F_\bullet$ of $G$ is a sequence of subgroups
\[
G=F_0=F_1\geq F_2\geq \ldots\geq F_r=\{1\}
\]
with the property that $[F_i,F_j]\subset F_{i+j}$ for every $i,j$. The lower central series $G_\bullet$ of a nilpotent group $G$ is a filtration \cite[Corollary 10.3.5]{hall}. Lemma \ref{lem:gen.coms} implies that the sequence $\overline G_\bullet$ of generalised commutator subgroups is also a filtration.

The set $\poly(H,F_\bullet)$ of \emph{polynomial maps} from a group $H$ into a nilpotent group $G$ with respect to a filtration
$F_\bullet$ of $G$ is then defined by
\[
\poly(H,F_\bullet)=\{f:H\to G \ \big| \ \p^{h_1}\cdots\p^{h_i}f(x)\in F_i\text{ for all $i$ and all $x, h_1,\ldots,h_i\in H$}\}.
\]
Note that  a function $f:G \to \C$ is a polynomial of degree at most $k$ according to our previous definition if and only if it is  a polynomial with respect to the filtration
\[
\C=F_0=F_1=   \ldots =F_k   \geq F_{k+1}=\{0\}
\]
Note also that if $G_\bullet$ is the lower central series and $F_\bullet$ is any other filtration then $G_i\subset F_i$ for each $i$, as is easily seen by induction. This implies in particular that
\begin{equation}\label{eq:filt.lcs}
\poly(H,G_\bullet)\subset\poly(H,F_\bullet).
\end{equation}
\begin{lemma}\label{lem:identity.poly}
Let $G$ be a finitely generated group, let $F_\bullet$ be a filtration of $G$, and let $\varphi:G\to G$ be the identity map. Then $\varphi\in\poly(G,F_\bullet)$.
\end{lemma}
\begin{proof}
By \eqref{eq:filt.lcs} it is sufficient to show that $\varphi\in\poly(G,G_\bullet)$. Given $y_1,y_2,\ldots\in G$, define $\alpha_1(y_1) = y_1$ and $\alpha_{k}(y_1,\ldots,y_{k}) = [y_{k}^{-1} , \alpha_{k-1}(y_1,\ldots,y_{k-1})^{-1} ]$ for $k=2,3,\ldots$.

We claim that for any $k \geq 1$ 
and any $y_1,\ldots,y_k \in G$ we have
$$ \p^{y_k} \cdots \p^{y_1} \varphi(x) =  
x \alpha_k(y_1,\ldots,y_k) x^{-1},  $$
which immediately implies the desired result, since $\alpha_{k}(y_1,\ldots,y_{k})\in G_k$.
We prove this claim by induction on $k$.
For $k=1$ we have $\p^y \varphi(x) = x y x^{-1}$ and the claim holds.
For $k>1$ we have, by induction,
\begin{align*}
\p^{y_k} \p^{y_{k-1}} \cdots \p^{y_1} \varphi(x) & =   \p^{y_{k-1}} \cdots \p^{y_1} \varphi(x y_k) 
\big( \p^{y_{k-1}} \cdots \p^{y_1} \varphi(x)  \big)^{-1} \\
& = x y_k    \alpha_{k-1}(y_1,\ldots,y_{k-1})  y_k^{-1} x^{-1} 
x \alpha_{k-1}(y_1,\ldots,y_{k-1})^{-1}   x^{-1} \\
& = x [ y_k^{-1} , \alpha_{k-1}(y_1,\ldots,y_{k-1})^{-1} ] x^{-1}\\
&= x \alpha_k(y_1,\ldots,y_k) x^{-1},
\end{align*}
as claimed.
\end{proof}

An important result about $\poly(H,F_\bullet)$ is that it is in fact a group under pointwise multiplication. This is due to Lazard \cite{lazard} when $G$ is a nilpotent Lie group, and to Leibman \cite[Proposition 3.4]{leibman} when $G$ is an arbitrary finitely generated nilpotent group. Leibman's result is stated in terms of the so-called ``lc-degree'', which corresponds to cases in which $F_\bullet$ is the lower central series or certain refinements of it. The following formulation is due to Green and Tao \cite{gt}.
\begin{prop}[Lazard; Leibman; Green--Tao {\cite[Proposition 6.2]{gt}}]\label{prop:llgt}
Let $G$ be a nilpotent group with a filtration $F_\bullet$, and let $H$ be a group. Then the set $\poly(H,F_\bullet)$ forms a group under the operation of taking pointwise products.
\end{prop}

We will deduce Proposition \ref{prop:coord.mon.poly} from the following slightly more general fact.
\begin{prop}\label{prop:leib.3.14}
Let $G$ be a torsion-free $\ell$-step nilpotent group with coordinate system $e_1,\ldots,e_{n_\ell}$, and let $H$ be a group. Suppose that $\varphi\in\poly(H,\overline G_\bullet)$. Then, writing $\varphi_1(h),\ldots,\varphi_{n_\ell}(h)$ for the coordinates of $\varphi(h)$ for $h\in H$, we have $\varphi_i\in P^{\sigma(i)}(H)$ for each $i$.
\end{prop}
We first isolate the following observation as a lemma for ease of later reference.
\begin{lemma}\label{lem:ab.coords}
Let $\psi:G\to\C^d$, and write $\psi_1,\ldots,\psi_d$ for the coordinate maps that make up $\psi$. Then $\psi\in P^k(G,\C^d)$ if and only if for each $i$ we have $\psi_i\in P^k(G)$.
\end{lemma}
\begin{proof}[Proof of Proposition \ref{prop:leib.3.14}]We essentially reproduce the proof of \cite[Proposition 3.14]{leibman} with $\overline G_i$ in place of $G_i$.
We prove by induction on $k=0,1,\ldots,\ell$ that if $\varphi\in\poly(H,\overline G_\bullet)$ satisfies  $\varphi(H)\subset\overline G_{\ell +1 -k}$ then  $\varphi_i\in P^{\sigma(i)}(H)$ for each $i$.
Note that this statement is trivial when $k=0$.

Suppose that $\varphi(H)\subset\overline G_{\ell +1 -k}$, and write $j=\ell +1 -k$.
Consider the map $\psi:H\to\overline G_j/\overline G_{j+1}$ induced by $\varphi$ via $\psi(h)=\varphi(h)\overline G_{j+1}$. Since $\varphi\in\poly(H,\overline G_\bullet)$, we have $\psi \in P^j(H,\overline G_j/\overline G_{j+1})$,
and so Lemma \ref{lem:ab.coords} implies that
\begin{equation}\label{eq:leib.1}
\varphi_{n_{j-1}+1},\ldots,\varphi_{n_j}\in P^j(H).
\end{equation}
Abusing notation slightly, we define the map $e_i^{\varphi_i}:H\to G$ by $h\mapsto e_i^{\varphi_i(h)}$. Since the image of $e_i^{\varphi_i}$ in $G$ generates an abelian subgroup (the cyclic subgroup generated by $e_i$), Lemma \ref{lem:ab.coords} and \eqref{eq:leib.1} imply that for each $i=n_{j-1}+1,\ldots,n_j$ we have $e_i^{\varphi_i}\in P^j(H,G)$. Since $e_i\in\overline G_j$, this implies in particular that
\[
e_{n_{j-1}+1}^{\varphi_{n_{j-1}+1}},\ldots,e_{n_j}^{\varphi_{n_j}}\in\poly(H,\overline G_\bullet).
\]
Since $\poly(H,\overline G_\bullet)$ is a group under pointwise multiplication by Proposition \ref{prop:llgt}, it follows that the map $\varphi'$ defined by
\[
\varphi'=\left(e_{n_{j-1}+1}^{\varphi_{n_{j-1}+1}} \cdot \ldots \cdot e_{n_j}^{\varphi_{n_j}} \right)^{-1}\varphi =e_{n_j+1}^{\varphi_{n_j+1}}\cdots e_{n_\ell}^{\varphi_{n_\ell}}
\]
belongs to $\poly(H,\overline G_\bullet)$. However, $\varphi'(H)\subset\overline G_{j+1}= \overline G_{\ell +1 -  (k-1)}$, and so the desired result follows from the induction hypothesis when $i>n_j$, and from \eqref{eq:leib.1} otherwise.
\end{proof}

\begin{proof}[Proof of Proposition \ref{prop:coord.mon.poly}]
The coordinate map factors through $G/\overline G_{\sigma(i)}$, so by Lemma \ref{lem:normal.factor} we may assume that $G$ is nilpotent and has no torsion.
In that case, if $\varphi:G\to G$ is the identity map then $\varphi_i(x)=x_i$, and so the result follows from Lemma \ref{lem:identity.poly} and Proposition \ref{prop:leib.3.14}.
\end{proof}

\begin{lemma}\label{lem:coords.mult.add}
Let $G$ be a finitely generated group with coordinate system $e_1,e_2,\ldots$, and let $x,u\in G$. Then the $j$th coordinate of $xu$ depends only on the first $j$ coordinates of $x$ and of $u$.
Moreover, if $u_j$ is the first non-zero coordinate of $u$ then the first $j$ coordinates of $xu$, and of $ux$, are $x_1,\ldots,x_{j-1},x_j+u_j$.
\end{lemma}
\begin{proof}
This follows from Lemma \ref{lem:gen.coms} and repeated use of the commutator identity $ab=ba[a,b]$
\end{proof}

\begin{proof}[Proof of Proposition \ref{thm:polys.equiv}]
Lemma \ref{lem:deg.prod} and Proposition \ref{prop:coord.mon.poly} imply that a coordinate polynomial of degree at most $k$ belongs to $P^k(G)$. We may therefore assume that $f\in P^k(G)$ and prove that $f$ is a coordinate polynomial of degree at most $k$.

Proposition \ref{prop:comm.deriv} implies that $f$ factors through $G/\overline G_{k+1}$, and so for every $x\in G$ the value of $f(x)$ depends only on the coordinates $x_1,\ldots,x_{n_k}$. If $f(x)$ does not depend on any of these coordinates then $f$ is constant, and hence a coordinate polynomial of degree $0$. We may therefore assume that $f(x)$ depends only on $x_1,\ldots,x_j$ and proceed by induction on $j$.

For each fixed $x_1,\ldots,x_{j-1}$ we may view $f$ as a function of $x_j$; thus there are functions $\alpha_{x_1,\ldots,x_{j-1}}:\Z\to\C$ such that
\[
f(x)=\alpha_{x_1,\ldots,x_{j-1}}(x_j).
\]
Lemma \ref{lem:coords.mult.add} implies that, for each $m$,
the first $j$ coordinates of $xe_j^m$ are $x_1,\ldots,x_{j-1},x_j+m$.  This in turn implies that
\begin{equation}\label{eq:partial}
(\partial_{e_j} )^mf(x)=(\partial_1)^m\alpha_{x_1,\ldots,x_{j-1}}(x_j).
\end{equation}
Proposition \ref{prop:comm.deriv} implies that $\partial_{e_j}$ reduces the degree of a polynomial by $\sigma(j)$, and so
\[
(\partial_{e_j} )^{\lfloor k/\sigma(j)\rfloor+1}f = 0
\]
It follows from \cite[Lemma 2.2]{leibman} that a function $p:\Z\to \C$ with $(\p_1)^m p = 0$ is a polynomial
of degree at most $m-1$.
In conjunction with (\ref{eq:partial}), 
all this implies that $\alpha_{x_1,\ldots,x_{j-1}}$ is a polynomial on $\Z$ of degree at most $\lfloor k/\sigma(j) \rfloor$.

It is an elementary fact that functions in $P^k(\Z)$ are ordinary polynomials of degree $k$ (see, for example, \cite[1.8]{leibman}). There are therefore a natural number $n\le k/\sigma(j)$ and real numbers $\alpha^{(i)}_{x_1,\ldots,x_{j-1}}$ for $0\le i\le n$ such that
\[
\alpha_{x_1,\ldots,x_{j-1}}(\xi)=\sum_{i=0}^n\alpha^{(i)}_{x_1,\ldots,x_{j-1}} \cdot \xi^i.
\]
Another application of (\ref{eq:partial}) therefore implies that there is a non-zero constant $c_n$ such that $(\partial_{e_j})^n f(x)=c_n\alpha^{(n)}_{x_1,\ldots,x_{j-1}}$. By Proposition \ref{prop:comm.deriv}, every application of $\partial_{e_j}$ reduces the degree of a polynomial by $\sigma(j)$.
It follows that  the function  $x \mapsto \alpha^{(n)}_{x_1,\ldots,x_{j-1}}$ belongs to $P^{k-\sigma(j)n}(G)$, and so is a coordinate polynomial of degree at most $k-\sigma(j)n$ by induction.

Applying the same argument to the function $x \mapsto f(x)-\alpha^{(n)}_{x_1,\ldots,x_{j-1}}\cdot x_j^n$, which also belongs to $P^k(G)$ by Lemma \ref{lem:deg.prod} and Proposition \ref{prop:coord.mon.poly}, we see that $\alpha^{(n-1)}_{x_1,\ldots,x_{j-1}}$ is a coordinate polynomial of degree at most $k-\sigma(j)(n-1)$. Continuing in this manner, each $\alpha^{(i)}_{x_1,\ldots,x_{j-1}} $ is a coordinate polynomial of degree at most $k-\sigma(j)i$.
We conclude that for any $0 \leq i \leq n$,
the function $x \mapsto \alpha^{(i)}_{x_1,\ldots,x_{j-1} } \cdot x_j^i$ is a coordinate polynomial of degree at most
$k-\sigma(j) i + \sigma(j) i = k$.
This proves that $f$ is a coordinate polynomial of degree at most $k$, as desired.
\end{proof}

\subsection{Dimensions of spaces of polynomials}
In this subsection we prove Proposition \ref{prop:dim.of.poly}, and use it to prove the following result.
\begin{prop}\label{prop:restriction}
Let $N$ be a finitely generated nilpotent group and let $H$ be a finite-index subgroup of $N$. Then the restriction map $\rho$ from $P^k(N)$ to functions on $H$ is a linear bijection between $P^k(N)$ and $P^k(H)$.
\end{prop}

\begin{proof}[Proof of Proposition \ref{prop:dim.of.poly}]
Lemma \ref{lem:tors.finite} implies that each $G_j$ has finite index in $\overline G_j$, and so $\overline G_j/\overline G_{j+1}$ has the same torsion-free rank as $G_j/G_{j+1}$ for each $j$. The non-negative solutions to \eqref{eq:integers_sols_dim} are therefore in bijection with the coordinate monomials of degree at most $k$ on $G$. These coordinate monomials span $P^k(G)$ by Proposition \ref{thm:polys.equiv}, and from the uniqueness of the representation of $x \in G$ in the form \eqref{eq:coords.def} they are linearly independent.
\end{proof}

Recall that  the \emph{iterated commutator} $[x_1,x_2,\cdots,x_j]_j$ is defined inductively by
\[
[x_1,x_2]_2=[x_1,x_2]=x_1^{-1}x_2^{-1}x_1x_2;\qquad[x_1,\ldots,x_j]_j=[[x_1,\ldots,x_{j-1}]_{j-1},x_j]\quad(j\ge3)
\]
\begin{lemma}\label{lem:comm.multilin}
Let $G$ be a group. Then the map
\[
G\times\cdots\times G \to G_j/G_{j+1}
\]
induced by the iterated commutator $[\,\,\, ,\,\cdots\, ,\,\,]_j$ is a homomorphism in each variable, the image of which is trivial when any variable belongs to $[G,G]$.
\end{lemma}
\begin{proof}
This is well known, essentially appearing in \cite[\S6]{bass}, for example. It follows from \cite[(10.2.1.2) \& (10.2.1.3)]{hall} that if $A$ and $B$ are subgroups of $G$ whose commutator $[A,B]$ lies in the centre of $G$ then the commutator map $A\times B\to[A,B]$ is a homomorphism in each variable. Applying this to $G_{j-1}/G_{j+1}$ and $G/G_{j+1}$, we see that the commutator map $G_{j-1}\times G\to G_{j}$ induces a map
\[
G_{j-1}\times G\to G_j/G_{j+1}
\]
that is a homomorphism in each variable and trivial on $G_j\times G$ and $G_{j-1}\times G_2$. The lemma therefore follows by induction on $j$.
\end{proof}

\begin{lemma}\label{lem:ab.fin.ind}
Let $N$ be a finitely generated abelian group with finite-index subgroup $H$. Let $Z$ be another abelian group, and suppose that $\alpha:N^j\to Z$ is a homomorphism in each variable. Define subgroups $\alpha(N),\alpha(H)$ of $Z$ via
\[
\alpha(N)=\langle\alpha(n_1,\ldots,n_j):n_i\in N\rangle,
\]
\[
\alpha(H)=\langle\alpha(h_1,\ldots,h_j):h_i\in H\rangle.
\]
Then $\alpha(H)$ has finite index in $\alpha(N)$.
\end{lemma}
\begin{proof}
Fix a generating set $e_1,\ldots,e_m$ for $N$, and write $r$ for the exponent of the finite abelian group $N/H$; thus $e_i^r\in H$ for every $i$. The group $\alpha(N)$ is generated by the finite set of elements of the form $\alpha(e_{i_1},\ldots,e_{i_j})$. By multilinearity of $\alpha$ each such element satisfies $\alpha(e_{i_1},\ldots,e_{i_j})^{r^j}\in\alpha(H)$, and so the quotient group $\alpha(N)/\alpha(H)$ is an abelian group generated by finitely many elements of finite order, and hence finite, and the lemma is proved.
\end{proof}
\begin{lemma}\label{lem:gen.com.fin.ind}
Let $N$ be a finitely generated $\ell$-step nilpotent group and let $H$ be a subgroup of finite index in $N$. Then for every $j\ge1$ the subgroup $H_j$ has finite index in $N_j$.
\end{lemma}
\begin{proof}
The conclusion of the lemma is trivial when $j>\ell$, so by induction on $\ell-j$ we may assume that $H_{j+1}$ has finite index in $N_{j+1}$.

Lemma \ref{lem:comm.multilin} implies that the iterated commutator $[\,\,\, ,\,\cdots\, ,\,\,]_j$ induces a multilinear map
\[
\alpha:N/[N,N] \times\cdots\times N/[N,N] \to N_j/N_{j+1}.
\]
It follows from \cite[Theorem 10.2.3]{hall} that, in the notation of Lemma \ref{lem:ab.fin.ind}, we have $N_j/N_{j+1}=\alpha(N/[N,N])$ and $H_j/(H_j\cap N_{j+1})=\alpha(H/(H\cap[N,N]))$, and so Lemma \ref{lem:ab.fin.ind} implies that $H_j/(H_j\cap N_{j+1})$ has finite index in $N_j/N_{j+1}$.

The inductive hypothesis implies in particular that $(H_j\cap N_{j+1})$ has finite index in $N_{j+1}$, and so $H_j$ has finite index in $N_j$, as claimed.
\end{proof}

\begin{proof}[Proof of Proposition \ref{prop:restriction}]
A polynomial of degree at most $k$ on $N$ certainly restricts to a unique polynomial of degree at most $k$ on $H$, and so $\rho$ does indeed map $P^k(N)$ into $P^k(H)$. By \cite[Proposition $4.2$]{leibman}, whenever $G$ is an amenable group and $f:G \to \C$ is a  non-zero polynomial, the preimage of $0$ under $f$ has zero density in $G$.  Recall that nilpotent groups are amenable. Being a finite-index subgroup, $H$ has positive density in $N$, and so it follows that $\rho$ is injective.
Lemma \ref{lem:gen.com.fin.ind} implies that $H_j/H_{j+1}$ has the same torsion-free rank as $N_j/N_{j+1}$, and so Proposition \ref{prop:dim.of.poly} implies that $\dim P^k(N)=\dim P^k(H)$, and so $\rho$ is in fact bijective.
\end{proof}

\begin{remark}
Proposition \ref{prop:restriction} does not necessarily hold if $N$ is not nilpotent. For example, the infinite dihedral group $D_\infty$ contains $\Z$ as a subgroup of index $2$, and the identity on $\Z$ is clearly a non-constant polynomial of degree $1$. However, it is well known that any  homomorphism from $D_\infty$ into (the additive group of) $\R$ is trivial. This easily implies that any polynomial on $D_\infty$ is constant, and in particular that the identity on $\Z$ does not extend to a polynomial on $D_\infty$.
\end{remark}

\subsection{Mal'cev bases and polynomials on Lie groups}\label{subsec:alex}
As we mentioned in the introduction,
Theorem \ref{thm:alex.cosets} is originally due to Alexopoulos \cite{alex}, with slight variations.
One obvious difference is that in the in the work of Alexopoulos polynomials are defined only for a torsion-free nilpotent group, by considering it as a lattice in a Lie group.
The purpose of this subsection is to describe the definition used by Alexopoulos and to show that, where it is defined, it coincides with ours.

An embedding theorem of Mal'cev \cite{rag} states that if a finitely generated nilpotent group $N$ 
has no torsion then it embeds as a discrete,
cocompact subgroup of a simply connected nilpotent Lie group $\mathcal N$ of the same nilpotency class, $\ell$, as $N$. Let $\n$ be the Lie algebra of $\mathcal N$, and write
\[
\n=\n_0=\n_1\supset\n_2\supset\cdots\supset\n_\ell\supset\n_{\ell+1}=\{0\}
\]
for the lower central series of $\n$. Write $d=\dim\n$, and for each $i$ write
\[
m_i=d-\dim\n_i.
\]
A related result of Mal'cev \cite{malcev} then states that there is a basis $X_1,\ldots,X_d$ for $\n$ that satisfies the following properties.
\begin{enumerate}
\renewcommand{\labelenumi}{(\roman{enumi})}
\item For each $j=1,\ldots,\ell$ we have $\n_j=\Span\{X_{m_j+1},\ldots,X_d\}$.
\item For every $x\in\mathcal N$ there is a unique $(x_1,\ldots,x_d)\in\R^{m_\ell}$ such that $x=\exp(x_1X_1)\cdots\exp(x_dX_d)$.
\item The subgroup $N$ of $\mathcal N$ consists precisely of those $x\in\mathcal N$ for which $(x_1,\ldots,x_d)\in\Z^d$.
\end{enumerate}
The basis $X_1,\ldots,X_d$ is called a \emph{Mal'cev basis compatible with $N$}, and the coordinates $(x_1,\ldots,x_d)$ for $x\in\mathcal N$ appearing in (ii) are called \emph{Mal'cev coordinates}, or \emph{coordinates of the second kind}. By identifying each $x\in\mathcal N$ with its $d$-tuple $(x_1,\ldots,x_d)$ of Mal'cev coordinates, we may identify $\mathcal N$ with $\R^d$. Moreover, by property (iii) above we may similarly identify $N$ with $\Z^d$.

Having made these identifications, Alexopoulos defines a \emph{polynomial} on $\mathcal N$ to be a polynomial, in the usual sense, on $\R^d$, and a \emph{polynomial} on $N$ to be the restriction of a polynomial on $\mathcal N$ (or, equivalently by (iii), a polynomial on $\Z^d$). This definition of a polynomial coincides with our Definition \ref{def:coord.poly} by the following result.
\begin{prop}\label{prop:Lie.poly.coord}
Let $N$ be a discrete cocompact subgroup of a simply connected nilpotent Lie group $\mathcal N$ with Lie algebra $\n$. Let $X_1,\ldots,X_d$ be a Mal'cev basis for $\n$ compatible with $N$. Then the elements $e_i=\exp X_i$ form a coordinate system for $N$, and the coordinates of an element with respect to $e_1,\ldots,e_d$ are precisely its Mal'cev coordinates with respect to $X_1,\ldots,X_d$.
\end{prop}
The proof of Proposition \ref{prop:Lie.poly.coord} is quite simple. However, we do need the following fact.
\begin{lemma}\label{lem:ej.in.N_j}
The generalised commutator subgroups $\overline N_j$ satisfy $N\cap\mathcal N_j=\overline N_j$.
\end{lemma}
\begin{proof}
Suppose $x\in\overline N_j$, so that there is some $n$ such that $x^n\in N_j$. This implies that $n\log x\in\n_j$, and hence that $\log x\in\n_j$, and so $x\in\mathcal N_j$. Conversely, $N_j$ is cocompact in $\mathcal N_j$ \cite{rag}, and so has finite index in $N\cap\mathcal N_j$. This implies that if $x\in N\cap\mathcal N_j$ then some power of $x$ belongs to $N_j$, as required.
\end{proof}
\begin{remark}
It is not hard to construct examples in which $N\cap\mathcal N_j\ne N_j$
\end{remark}
\begin{proof}[Proof of Proposition \ref{prop:Lie.poly.coord}]
The only part of the proof that is not straightforward is the requirement that the elements $e_{n_{i-1}+1},\ldots,e_{n_i}$ generate $\overline N_i/\overline N_{i+1}$. However, this is immediate from Lemma \ref{lem:ej.in.N_j} and property (i) of the Mal'cev basis.
\end{proof}
\begin{remark}
In \cite{alex} the role of degree is played by the so-called \emph{homogeneous degree} of a polynomial on a torsion-free nilpotent group. Lemma \ref{lem:ej.in.N_j} implies that it is equivalent to the notion of degree given in Definition \ref{def:coord.poly}.
\end{remark}

\section{The image of the Laplacian}\label{subsec:image}
In this section we prove Theorem \ref{thm:lap_poly_img}, which states that
\begin{equation}\label{eq:heilbronn}
\Delta(P^{k+2}(G))= P^k(G)
\end{equation}
for every finitely generated group $G$. One inclusion is quite elementary, as follows.
\begin{prop}
\label{prop:laplacian.image}
For any $k\geq 0$ the Laplacian operator satisfies
\[
\Delta(P^{k+2}(G))\subset P^k(G).
\]
\end{prop}

\begin{proof}
Let $p\in P^k(G)$. As we remarked earlier, because $\mu$ is smooth and $p$ has polynomial growth, the sum
$$\E_s p(x) = \sum_s \mu(s) p(xs)$$
converges absolutely.

By the symmetry of $\mu$ we may write
\begin{align*}
\Delta p(x)&=\textstyle\frac{1}{2}\E_s[2p(x)-p(xs)-p(xs^{-1})]\\
&=-\textstyle\frac{1}{2}\E_s[\partial^s\partial^sp(xs^{-1})].
\end{align*}
Corollary \ref{cor:P^k.G.invariant} implies that $\partial^s\partial^s p(xs^{-1})$ is a polynomial of degree at most $k-2$ in $x$, which implies the desired inclusion.
\end{proof}

The second inclusion of \eqref{eq:heilbronn} is slightly trickier.
\begin{prop}\label{prop:lap.image.surj}
For any $k\geq 0$ the Laplacian operator satisfies
\[
\Delta(P^{k+2}(G))\supset P^k(G).
\]
\end{prop}

It will be convenient to consider polynomials of the form $x_1^k$ separately. The following is essentially a special case of the result Heilbronn used to prove \cite[Theorem 3]{heilbronn}.
\begin{lemma}\label{lem:base}
Let $p(x)=x_1^k$ with $k\ge0$. Then there exists a polynomial $\hat p$ of degree $k+2$ such that $\hat p(x)$ depends only on $x_1$, and such that $\Delta\hat p=p$.
\end{lemma}
\begin{proof}
Define $q(x)=x_1^{k+2}$. The map $x \mapsto x_1$ is a homomorphism into an abelian group, which means that for every $s\in G$ we have $(xs)_1=x_1+s_1$, and so the symmetry of $\mu$ gives
\begin{align*}
\Delta q(x)&=\textstyle\frac{1}{2}\E_s[2x_1^{k+2}-(x_1+s_1)^{k+2}-(x_1-s_1)^{k+2}]\\
     &=-\E_s[s_1^2]{k+2\choose2}x_1^k+r(x)
\end{align*}
for some polynomial $r$ of degree at most $k-1$ such that $r(x)$ depends only on $x_1$. Note that the fact that the support of $\mu$ generates $G$ implies that there is some $s\in G$ with $\mu(s)>0$ and $s_1\ne0$, and so $\E_s[s_1^2]\ne0$.

We now prove the lemma by induction on $k$. If $k=0$ then $r = 0$ and the lemma follows easily. For $k>1$ the induction hypothesis implies that there is a polynomial $\hat r$ of degree at most $k+1$ such that $\hat r(x)$ depends only on $x_1$ and such that $\Delta\hat r=r$, and so we have
\[
\Delta(q-\hat r)(x)=\E_s[s_1^2]{k+2\choose2}x_1^k,
\]
and the lemma is proved.
\end{proof}

We will deduce Proposition \ref{prop:lap.image.surj} from the following more precise statement.
\begin{prop}\label{prop:heil.surj}
For every coordinate monomial $p$ of degree $k$ on $G$ such that $p(x)$ depends only on the coordinates $x_1,\ldots,x_m$, there exists some polynomial $\hat p$ of degree $k+2$ such that $\hat p(x)$ depends only on $x_1,\ldots,x_m$ and such that $\Delta\hat p=p$.
\end{prop}

\begin{proof}
We proceed by induction on $m$, noting that the case $m=1$ follows from Lemma \ref{lem:base}. When $m>1$, the monomial $p$ is of the form $qr$, with $r(x)=x_m^n$ and $q(x)$ depending only on those $x_i$ with $i<m$. We prove this case by a second induction, this time on $n$, the case $n=0$ following from the $m$-induction hypothesis. For $n\ge1$, the $m$-induction hypothesis implies that there exists some $\hat q$ of degree at most $\deg q+2$ such that $\hat q(x)$ depends only on $x_1,\ldots,x_{m-1}$, and such that $\Delta\hat q=q$.

Now note that
\begin{align*}
\Delta(\hat qr)(x)&=\E_s[\hat q(x)r(x)-\hat q(xs)r(xs)]\\
      &=\E_s[(\hat q(x)-\hat q(xs))r(x)+\hat q(xs)(r(x)-r(xs))]\\
      &=\Delta\hat q(x) \cdot r(x)-\E_s[\hat q(xs)\partial^sr(x)]\\
      &=q(x)r(x)-\E_s[\hat q(xs)\partial^sr(x)]\\
      &=p(x)-\E_s[\hat q(xs)\partial^sr(x)].
\end{align*}
The expression $\partial^sr(x)$ depends only on $x_1,\ldots,x_m$ by Lemma \ref{lem:coords.mult.add}, whilst the highest power of $x_m$ featuring in any term of $\partial^sr(x)$ is less than $n$ by  Proposition \ref{thm:polys.equiv}. Since $\hat q(x)$ depends only on $x_1,\ldots,x_{m-1}$, this implies that $\E_s[\hat q(xs)\partial^sr(x)]$ depends only on $x_1,\ldots,x_m$, and that the highest power of $x_m$ appearing in any of its terms is less than $n$. Moreover, since $\deg(\hat qr)=k+2$ and $\deg p=k$, Proposition \ref{prop:laplacian.image} implies that the polynomial $\E_s[\hat q(xs)\partial^sr(x)]$ has degree at most $k$.

The $n$-induction hypothesis therefore implies that there is some polynomial $v$ of degree at most $k+2$ such that $v(x)$ depends only on $x_1,\ldots,x_m$ and such that $\Delta v(x)=\E_s[\hat q(xs)\partial^sr(x)]$. It follows that taking $\hat p=\hat qr+v$ satisfies the proposition.
\end{proof}
\begin{proof}[Proof of Proposition \ref{prop:lap.image.surj}]
Proposition \ref{prop:comm.deriv} implies that if $p\in P^k(G)$ then $p(x)$ depends only on the coordinates $x_1,\ldots,x_{n_k}$, and so the desired inclusion follows from Proposition \ref{prop:heil.surj}.
\end{proof}
\begin{proof}[Proof of Theorem \ref{thm:lap_poly_img}]
The theorem is immediate from Propositions \ref{prop:laplacian.image} and \ref{prop:lap.image.surj}.
\end{proof}

\section{Conclusion of  the main results:  passing to a finite-index subgroup}\label{subsec:nilp.reduc}
In this section we deduce Theorems \ref{thm:alex.cosets}  and \ref{thm:dim} from Theorems \ref{thm:alex.gen} and \ref{thm:lap_poly_img}.

\begin{prop}[{\cite[Proposition 5.4]{mey-yad}}]\label{prop:dim}
Let $G$ be a finitely generated group with a subgroup $H$ of finite index. Let $\mu$ be a probability measure on $G$ satisfying the standing assumptions.
Then there is a probability measure $\mu_H$ on $H$ satisfying the standing assumptions such that the restriction map from functions on $G$ to functions on $H$ defines a linear bijection between $H^k(G,\mu)$ and $H^k(H,\mu_H)$.
\end{prop}

\begin{proof}[Proof of Theorem \ref{thm:dim}]
Kleiner's theorem \cite{kleiner} implies that $\dim H^k(G,\mu)<\infty$, and hence Theorem \ref{thm:alex.gen} implies that there is \emph{some} finite-index nilpotent subgroup $H$ of $G$ such that the restriction to $H$ of any $f\in H^k(G,\mu)$ is a polynomial of degree $k$. Proposition \ref{prop:dim}
therefore implies that $H^k(H,\mu_H)$ is the kernel of the Laplacian $\Delta_{\mu_H}$ applied to $P^k(H)$, and so Theorem \ref{thm:lap_poly_img} and Proposition \ref{prop:dim} imply that
$$\dim H^k(G,\mu) = \dim P^k(H) - \dim P^{k-2}(H).$$
Now simply note that $H \cap N$ is a finite-index subgroup of both $H$ and $N$, and hence that
Proposition \ref{prop:restriction} gives $\dim P^j(H)= \dim P^j(N)$ for every $j$.
\end{proof}

\begin{remark}
Strictly speaking, Kleiner's theorem states that $\dim H^k(G,\mu)<\infty$ when $G$ is finitely generated group of polynomial growth and 
$\mu = \frac{1}{|S|}\sum_{s \in S} \delta_s$ is the uniform probability measure on a finite symmetric generating set $S$. A relatively straightforward modification of the proof gives the same result for a finitely supported symmetric measure $\mu$, as used in the proof of Theorem \ref{thm:dim}. A similar generalisation in a different direction appears in \cite{bdky}.
\end{remark}

\begin{proof}[Proof of Theorem \ref{thm:alex.cosets}]
Let $\mu_N$ be the measure on $N$ given by Proposition \ref{prop:dim}. Theorem \ref{thm:dim} implies that $\dim H^k(N,\mu_N)=\dim P^k(N)-\dim P^{k-2}(N)$. Theorem \ref{thm:lap_poly_img} and Proposition \ref{prop:growth.of.polys}, on the other hand, imply that $\dim\left(H^k(N,\mu_N)\cap P^k(N)\right)$ also equals $\dim P^k(N)-\dim P^{k-2}(N)$. This implies that $H^k(N,\mu_N)\cap P^k(N)  = H^k(N,\mu_N)$ and so Proposition \ref{prop:dim} implies that every $f\in H^k(G,\mu)$ restricts to an element of $P^k(N)$.

To see that such an $f$ is a polynomial of degree $k$ with respect to $N$ on all of $G$, we must show more generally that for every $t\in G$ the function $N\to\R$ defined by $x\mapsto f(xt)$ is a polynomial of degree $k$. The function $p_t=t^{-1}f$ belongs to $H^k(G,\mu)$, and so by the first part of the theorem it restricts to a polynomial of degree at most $k$ on the nilpotent subgroup $t^{-1}Nt$. In particular, for every $x\in N$ and every $u_1,\ldots,u_{k+1}\in N$ we have
\[
\partial_{t^{-1}u_1t}\cdots\partial_{t^{-1}u_{k+1}t}p_t(t^{-1}xt)=0,
\]
and so the function $q_t:N\to\R$ defined by $q_t(x)=p_t(t^{-1}xt)$ is a polynomial of degree at most $k$ on $N$. However, $q_t(x)=p_t(t^{-1}xt)=t^{-1}f(t^{-1}xt)=f(xt)$, and so $f$ is indeed a polynomial with respect to $N$ on $G$.

Finally, $q_t$ factors through $N/N_{k+1}$ by Lemma \ref{lem:comm.deriv}, and so $N_{k+1}$ acts trivially on $f$ from the left, as required.
\end{proof}

\section{Additional corollaries}\label{sec:cor}

\subsection{Asymptotics as $k\to\infty$}\label{subsec:hj}
The purpose of this subsection is to prove Corollary \ref{cor:hj}. We make repeated use of the standard and easily verified computation
\[
\dim P^k(\Z^d)-\dim P^{k-1}(\Z^d)={d+k-1\choose k},
\]
which in turn implies that
\begin{equation}\label{eq:hom.polys'}
k^{d-1}\ll_d\dim P^k(\Z^d)-\dim P^{k-1}(\Z^d)\ll_dk^{d-1}
\end{equation}
for $k\ge1$. It will also be convenient to introduce some new notation. Let $x_1,\ldots,x_d$ be real variables, and let $\sigma_1,\ldots,\sigma_d$ be positive integers. Then we define the space
\[
\overline P^k(x_1,\ldots,x_d;\sigma_1,\ldots,\sigma_d)
\]
to be the subspace of polynomials in the $x_i$ spanned by the monomials $x_1^{a_1}\cdots x_d^{a_d}$ such that $\sigma_1a_1+\ldots+\sigma_da_d=k$.
\begin{lemma}\label{lem:hom}
For every $k$ we have
\begin{enumerate}
\renewcommand{\labelenumi}{(\roman{enumi})}
\item$\dim\overline P^k(x_1,\ldots,x_d;1,\sigma_2,\ldots,\sigma_d)\ge\dim\overline P^k(x_1,\ldots,x_d;1,\sigma_2,\ldots,\sigma_{d-1},\sigma_d+1)$ and
\item$\dim\overline P^k(x_1,\ldots,x_d;1,\sigma_2,\ldots,\sigma_d)\ge\dim\overline P^{k-1}(x_1,\ldots,x_d;1,\sigma_2,\ldots,\sigma_d)$,
\end{enumerate}
and for every $r,\ell$ we have
\begin{enumerate}
\renewcommand{\labelenumi}{(\roman{enumi})}
\setcounter{enumi}{2}
\item$\dim\overline P^{\ell r}(x_1,\ldots,x_d;1,\ell,\ldots,\ell)\ge\dim\overline P^r(x_1,\ldots,x_d;1,\ldots,1)$.
\end{enumerate}
\end{lemma}
\begin{proof}
In each case we exhibit an injection from the defining basis of the second space into the defining basis of the first space. Specifically, we take the maps
\begin{enumerate}
\renewcommand{\labelenumi}{(\roman{enumi})}
\item$x_1^{a_1}\cdots x_d^{a_d}\mapsto x_1^{a_1+a_d}x_2^{a_2}\cdots x_d^{a_d}$;
\item$x_1^{a_1}\cdots x_d^{a_d}\mapsto x_1^{a_1+1}x_2^{a_2}\cdots x_d^{a_d}$; and
\item$x_1^{a_1}\cdots x_d^{a_d}\mapsto x_1^{la_1}x_2^{a_2}\cdots x_d^{a_d}$.
\end{enumerate}
\end{proof}
\begin{remark}
It is not hard to construct examples showing that the assumption that one of the $\sigma_i$ is equal to $1$ is necessary in parts (i) and (ii) of Lemma \ref{lem:hom}.
\end{remark}
\begin{proof}[Proof of Corollary \ref{cor:hj}]
The upper bound follows from Theorem \ref{thm:dim} by (\ref{eq:hom.polys'}) and Lemma \ref{lem:hom} (i). The lower bound follows from Theorem \ref{thm:dim} by (\ref{eq:hom.polys'}) and all three parts of Lemma \ref{lem:hom}, using the fact that for every $\ell$-step nilpotent group $N$ of rank $d$ we have $\ell\le d$ and $\sigma(i)\le\ell$ for every $i$. In each case we also use the fact that $\sigma(1)=1$ for every nilpotent group $N$.
\end{proof}

\subsection{More on the structure of $H^k(G,\mu)$ in the virtually nilpotent case}

Theorem \ref{thm:alex.cosets} can be stated as follows. Let $G$ be a finitely generated group with a probability measure $\mu$ that satisfies the standing assumptions and is finitely supported, and let $N$ be a nilpotent subgroup of finite index in $G$.
Let $T \subset G$ be a \emph{right-transversal} for $N$ in $G$, so that every element of $G$ can be uniquely written in the form $xt$ with $t \in T$ and $x \in N$.
Then every element $f\in H^k(G,\mu)$ is of the form
\begin{equation}\label{eq:polys.on.G}
f(xt)= p_t(x),
\end{equation}
with each $p_t \in P^k(N)$.

Alexopoulos \cite{alex} also gives an additional piece of information, which we now reproduce for completeness.

\begin{prop}[Alexopoulos]\label{prop:alex_cosets}
If $f \in H^k(G,\mu)$ is of the form \eqref{eq:polys.on.G} then for any $t_1,t_2 \in T$ we have $p_{t_1} - p_{t_2} \in P^{k-1}(N)$.
\end{prop}

\begin{proof}
Since $p_t \in P^k(N)$, for every $t\in T$ and every $x_1,\ldots,x_k \in N$ the function
$\p_{x_1}\ldots \p_{x_k} p_t$ is constant  on $N$.
It follows that $N$ acts trivially from the left on $\p_{x_1}\ldots \p_{x_k} f$, and in particular that $\p_{x_1}\ldots \p_{x_k} f$ takes at most finitely many different values. By the maximum principle, the $\mu$-harmonic function $\p_{x_1}\ldots \p_{x_k} f$ must therefore be constant on $G$. In particular, for any $t_1,t_2 \in T$ we have $\p_{x_1}\cdots \p_{x_k}(p_{t_1}-p_{t_2})=0$,
and so $p_{t_1} - p_{t_2} \in P^{k-1}(N)$.
\end{proof}

\subsection{Groups with $\dim H^k(G,\mu) < \infty$}\label{subsec:gg}
In this subsection we prove Corollary \ref{cor:dim.indep.of.mu}.
Note that if $\pi:G \to \Gamma$ is a homomorphism and $\mu$ is a probability measure on $G$ satisfying the standing assumptions then the measure $\mu \circ \pi^{-1}$ on $\pi(G)$ also satisfies the standing assumptions, and
\begin{equation}\label{eq:proj.measure}
\dim H^k(G,\mu)\ge\dim H^k(\pi(G),\mu \circ \pi^{-1}).
\end{equation}

\begin{proof}[Proof of Corollary \ref{cor:dim.indep.of.mu}]Just as in the proof of Theorem \ref{thm:alex.gen}, the action of $G$ on $H^k(G,\mu)$ defines a homomorphism $\pi:G \to GL_n(\C)$ whose image $\pi(G)$ is virtually nilpotent. An inspection of the proof of Theorem \ref{thm:dim} reveals that it applies also to a non-finitely supported probability measure, provided it satisfies the standing assumptions and results in a finite-dimensional space of harmonic functions of polynomial growth of degree at most $k$. In particular, this implies that $\dim H^k(\pi(G),\nu \circ \pi^{-1})=\dim H^k(\pi(G),\mu \circ \pi^{-1})=\dim H^k(G,\mu)$. It therefore follows from (\ref{eq:proj.measure}) that $\dim H^k(G,\nu)\ge\dim H^k(G,\mu)$, and so the equality of dimensions follows by symmetry.
\end{proof}

\subsection{The kernel of  $G$ acting on $H^k(G,\mu)$}\label{subsec:kernel}
\begin{proof}[Proof of Corollary \ref{cor:kernel}]
If $k<\ell$ then $K$ is infinite by Theorem \ref{thm:alex.cosets} and Lemma \ref{lem:comm.deriv}.
To prove the converse, it is sufficient to show that $N\cap K$ is finite. By Proposition \ref{prop:dim} there is a smooth, symmetric generating probability measure $\mu_N$ on $N$
such that the restriction map defines a bijection between $H^\ell(G,\mu)$ and $H^\ell(N,\mu_N)$, and so we may assume that $G=N$. We claim then that $K$ is contained in $t(N)$, the subgroup of torsion elements of $N$, which is finite for finitely generated $N$
(see \cite{Robinson}, for example).

Let $y\in N$ be such that $y\notin t(N)$, and note that this is equivalent to saying that $y\notin\overline N_{\ell+1}$.
This implies in particular that there exists a smallest coordinate $j$ such that $y_j\ne 0$, and by Lemma \ref{lem:coords.mult.add} this coordinate satisfies
\begin{equation}\label{eq}
(yx)_i=x_i\qquad(i<j)\qquad\qquad\text{and}\qquad\qquad(yx)_j=x_j+y_j
\end{equation}
for every $x\in N$.

Define the polynomial $q$ on $N$ by $q(x)=x_j$, and note that $\deg q=\sigma(j)\le\ell$. It follows from Lemma \ref{lem:coords.mult.add} that $\Delta q(x)$ depends only on the coordinates $x_1,\ldots,x_j$.
However, it also follows from Proposition \ref{prop:laplacian.image} that $\Delta q$ is of degree at most $\sigma(j)-2$, and so no term of $\Delta q$ has an $x_j$ factor and $\Delta q$ therefore depends only on the coordinates $x_1,\ldots,x_{j-1}$.

Proposition \ref{prop:heil.surj} therefore implies that there is some polynomial $r$
of degree at most $\sigma(j)$ such that $r(x)$ depends only on the coordinates $x_1,\ldots,x_{j-1}$,
and such that $\Delta r=\Delta q$. Thus the polynomial $p$ defined by $p(x)=x_j-r(x)$ is harmonic and of degree at most $\sigma(j)$, which is at most $\ell$.
However, $p(yx)=p(x)+y_j$ by (\ref{eq}), and so $y\notin K$ and the corollary is proved.
\end{proof}

\end{document}